\scrollmode

\documentclass{amsart}
\usepackage{amssymb}
\usepackage[all]{xy}

\theoremstyle{plain}
\newtheorem{thm}{Theorem}
\newtheorem{cor}[thm]{Corollary}
\newtheorem{fact}[thm]{Fact}
\newtheorem{lem}[thm]{Lemma}
\newtheorem{prop}[thm]{Proposition}
\newtheorem{ques}{Question}

\newtheorem*{thmA}{Theorem A}
\newtheorem*{corB}{Corollary B}
\newtheorem*{corC}{Corollary C}
\newtheorem*{thmD}{Theorem D}

\theoremstyle{remark}
\newtheorem{rem}[thm]{Remark}

\theoremstyle{definition}
\newtheorem{example}{Example}

\numberwithin{thm}{section}
\numberwithin{example}{section} 
\numberwithin{equation}{section}

\DeclareMathOperator{\ccd}{cd}
\DeclareMathOperator{\Gal}{Gal}

\DeclareMathOperator{\Ext}{Ext}
\DeclareMathOperator{\Tor}{Tor}
\DeclareMathOperator{\Hom}{Hom}
\DeclareMathOperator{\diag}{diag}
\DeclareMathOperator{\op}{op}
\DeclareMathOperator{\gr}{gr}
\DeclareMathOperator{\rk}{rk}
\DeclareMathOperator{\ob}{ob}
\DeclareMathOperator{\ord}{ord}
\DeclareMathOperator{\ab}{ab}
\newcommand{\Cl}{\mathrm{Cl}}

\DeclareMathOperator{\image}{im}
\DeclareMathOperator{\kernel}{ker}

\DeclareMathOperator{\iid}{id}
\DeclareMathOperator{\Repa}{Re}

\newcommand{\F}{\mathbb F}
\newcommand{\N}{\mathbb N}
\newcommand{\Z}{\mathbb Z}
\newcommand{\C}{\mathbb C}
\newcommand{\R}{\mathbb R}
\newcommand{\Q}{\mathbb Q}

\newcommand{\bA}{\mathbf{A}}

\newcommand{\boT}{\mathbf{T}}
\newcommand{\bL}{\mathbf{L}}

\newcommand{\caU}{\mathcal{U}}
\newcommand{\dbr}{]\!]}
\newcommand{\dbl}{[\![}

\newcommand{\bor}{\mathbf{r}}
\newcommand{\Amod}{{}_{\bA}\mathbf{mod}}

\newcommand{\Atf}{{}_{\bA}\mathbf{mod}^{\mathrm{ft}}}

\newcommand{\bp}{\mathbf{p}}

\newcommand{\euX}{\mathfrak{X}}
\newcommand{\euY}{\mathfrak{Y}}
\newcommand{\caE}{\mathcal{E}}
\newcommand{\caP}{\mathcal{P}}
\newcommand{\caO}{\mathcal{O}}
\newcommand{\caM}{\mathcal{M}}

\newcommand{\boh}{\mathbf{h}}
\newcommand{\eps}{\varepsilon}
\newcommand{\der}{\partial}

\newcommand{\tchi}{\tilde{\chi}}
\newcommand{\brho}{\bar{\rho}}
\newcommand{\argu}{\hbox to 7truept{\hrulefill}}

\begin{document}
\title[Graded Lie algebras of type FP]{Graded Lie algebras of type FP}
\author{Th. Weigel}
\date{\today}
\address{Universit\`a di Milano-Bicocca\\
U5-3067, Via R.Cozzi, 53\\
20125 Milano, Italy}
\email{thomas.weigel@unimib.it}

\begin{abstract}
It will be shown that every $\N$-graded Lie algebra generated in degree $1$ of type FP
with entropy less or equal to 1 must be finite-dimensional (cf. Thm.~A). As a consequence
every Koszul Lie algebra with entropy less or equal to $1$ must be abelian (cf. Cor.~C).
These results are obtained
from a generalized Witt formula (cf. Thm.~D) for $\N$-graded Lie algebras of type FP
and the analysis of necklace polynomials at roots of unity. 
\end{abstract}


\maketitle

\section{Introduction}
\label{s:intro}
Several deep results in group theory relate
certain growth phenomena to the structure theory of a group,
e.g., M.~Gromov's celebrated theorem states that a finitely generated
group has polynomial word growth if, and only if, it is virtually nilpotent
(cf. \cite{grom:growth}); while A.~Lubotzky and A.~Mann showed  
that a finitely generated pro-$p$ group has polynomial subgroup growth
if, and only if, it is $p$-adic analytic (cf. \cite{luma:growth}).
In a second paper together with D.~Segal they achieved the beautiful result
that a finitely generated residually finite (discrete) 
group has polynomial subgroup growth if, and only if,
it is virtually soluble of finite rank (cf. \cite{lumase:growth}).

The main purpose of this paper is to establish an analogue of the just mentioned
results in the context of $\N$-graded $\F$-Lie algebras 
which are generated in degree $1$ and which are of finite type.
Following \cite{szsh:ent} for such a Lie algebra 
$\bL=\coprod_{k\geq 1}\bL_k$ one calls the number
\begin{equation}
\label{eq:entdef}
\boh(\bL)=\textstyle{\limsup_{n\to\infty}\sqrt[n]{\dim(\bL_n)}\in\R_{\geq 1}\cup\{0,\infty\}}
\end{equation}
the {\it entropy} of $\bL$. Moreover, a Lie algebra
$\bL$ is called to be {\it of type FP$_\infty$},
if the trivial left $\bL$-module $\F$ has a projective resolution $(P_\bullet,\der_\bullet,\eps)$,
where $P_k$ is a finitely generated projective left $\caU(\bL)$-module for all $k$. Here $\caU(\bL)$
denotes the {\it universal enveloping algebra} of $\bL$. 
The Lie algebra $\bL$ is said to be of {\it finite cohomological dimension} $\ccd(\bL)<\infty$,
if $\F$ has a projective resolution $(P_\bullet,\der_\bullet,\eps)$ of finite length,
i.e., there exists a positive integer $n$ such that $P_k=0$ for all $k\geq n$.
If $\bL$ is of type FP$_\infty$ and of finite cohomological dimension, $\bL$ is called to
be {\it of type FP} (cf. \cite[Chap.~VIII.6]{brown:coho}). The main result of this paper
can be formulated as follows (cf. Thm.~\ref{thm:grhU1}).

\begin{thmA}
Let $\bL$ be an $\N$-graded Lie algebra  
generated in degree $1$ and of type FP satisfying $\boh(\bL)\leq 1$.
Then $\bL$ is finite-dimensional and nilpotent, i.e., $\boh(\bL)=0$.
\end{thmA}

There are many examples of $\N$-graded Lie algebras of finite type which are generated in degree $1$
and whose entropy is equal to $1$, e.g., if $\bL$ is infinite dimensional and filiform (cf. \cite{mill:filiform})
one has $\dim(\bL_k)=1$ for $k\geq 2$ and therefore $\boh(\bL)=1$. 
The Lie algebras $\bL$ constructed in \cite{kkm:inter}
satisfy $\boh(\bL)=1$ and the series $(\dim(\bL_k))_{k\geq 1}$ has intermediate growth. 
As a consequence of Theorem A one obtains the following.

\begin{corB}
Let $\bL$ be an $\N$-graded finitely generated Lie algebra which is 
generated in degree $1$  satisfying $\boh(\bL)= 1$. Then either
$\ccd(\bL)=\infty$ or there exists $k\geq 2$ such that $\dim(H^k(\bL,\F))=\infty$.
\end{corB}

An $\N$-graded Lie algebra $\bL$ will be said to be {\it Koszul},
if $\caU(\bL)$ is a Koszul algebra. Koszul algebras are a rather
mysterious class of $\N_0$-graded associative algebras,
and there are still many open questions.
Theorem~A provides an answer to one of these open questions (cf. \cite[Chap.~7.1, Conj.~3]{popo:quad})
in case that the Koszul algebra $\bA$
happens to be the universal enveloping algebra of an $\N$-graded Lie algebra
(cf. Prop.~\ref{prop:koslieen1}).

\begin{corC}
Let $\bL$ be a Koszul Lie algebra satisfying $\boh(\bL)\leq 1$.
Then $\bL$ is finite-dimensional and abelian.
\end{corC}

An $\N$-graded Lie algebra $\bL$ of type FP has 
a {\it characteristic polynomial} (cf. \eqref{eq:charp1})
\begin{equation}
\label{eq:charpol}
\chi_{\bL}(y)=\textstyle{\prod_{1\leq i\leq n}(1-\lambda_i y)\in\C[y]}
\end{equation}
which depends entirely on the cohomology of $\bL$.
The complex numbers $\lambda_i$ will be called the 
{\it eigenvalues} of $\bL$.
For such a Lie algebra one has the following
generalized Witt formula (cf. Thm.~\ref{thm:witt}).

\begin{thmD}
Let $\bL$ be an $\N$-graded Lie algebra of type FP, and let
$\lambda_1,\ldots,\lambda_n\in\C$ be the eigenvalues of $\bL$.
Then
\begin{equation}
\label{eq:genWittForm}
\dim(\bL_k)=\textstyle{\sum_{1\leq i\leq n} M_k(\lambda_i),}
\end{equation} 
where 
$\mu\colon\N\to\Z$ is the M\"obius function and
$M_k(y)=\frac{1}{k}\sum_{j|k} \mu(k/j)\cdot y^j\in\Q[y]$ denotes the
necklace polynomial of degree $k$.
\end{thmD}

Theorem~A can be deduced from Theorem~D and an analysis of necklace
polynomials at roots of unity (cf. \S\ref{ss:neckroot}).
For an infinite dimensional $\N$-graded Lie algebra $\bL$ of type FP
which is generated in degree $1$ its entropy $\boh(\bL)$ is the positive
real root of the characteristic polynomial $\chi_{\bL}(y)$
of maximal absolute value (cf. Prop.~\ref{prop:lmaxent}, Cor.~\ref{cor:brez1}).
Some general results on the eigenvalues of $\chi_{\bL}(y)$ can be obtained
(cf. Remark~\ref{rem:FP}, Fact~\ref{fact:desrule}), but
several questions will remain unanswered (cf. Question \ref{ques:reparteig} and \ref{ques:turan}).
Since it seems that the Koszul property has been investigated only scarsely in the context of Lie
algebras, we close this paper with a discussion of three classes of examples illustrating the
diversity of this class of Lie algebras.
\vskip12pt

\noindent
{\bf Acknowledgement:} The author would like to thank P.~Spiga
for very helpful discussions concerning the roots of independence polynomials
of finite loop-free graphs (cf. Ex.~\ref{ex:RAL}(d)).


\section{Graded connected algebras of finite type}
\label{s:graded}
Throughout the paper $\F$ will denote a field.
By $\N_0$ (resp. $\N$) we denote the set of non-negative (resp. positive) integers,
i.e., $\N_0=\N\cup\{0\}$. 
An $\N_0$-graded $\F$-vector space
$V=\coprod_{k\geq 0} V_k$ is said to be
of {\it finite type}, if $\dim(V_k)<\infty$ for all $k\geq 0$. For such a graded $\F$-vector space
its {\it Hilbert series} is defined by
\begin{equation}
\label{eq:HilserV}
h_{V}(y)=\textstyle{\sum_{k\geq 0} \dim(V_k)\cdot y^k\in\Z\dbl y\dbr},
\end{equation}
e.g., $V$ is finite-dimensional if, and only if,
$h_{V}(y)$ is a polynomial.


\subsection{Graded connected algebras of finite type}
\label{ss:gradedFT}
An {\it $\N_0$-graded $\F$-algebra} $\bA=\coprod_{k\geq 0} \bA_k$
is an associative $\F$-algebra and an $\N_0$-graded $\F$-vector space
satisfying $\bA_m\cdot\bA_n\subseteq\bA_{m+n}$ for all $m,n\geq 0$.
From now we will omit the appearance of the field $\F$ in the notation.

An $\N_0$-graded algebra $\bA$ is called to be {\it connected}, if $\bA_0=\F\cdot 1$, i.e.,
in this case one has a unique {\it augmentation} $\eps\colon\bA\to\F$
which is a homomorphism of $\N_0$-graded algebras.
By $\bA^+=\coprod_{k\geq 1}\bA_k=\kernel(\eps)$  we denote its
{\it augmentation ideal}, and $\F$ will denote the $\N_0$-graded left $\bA$-module
whose representation is equal to $\eps$.

For an $\N_0$-graded, connected algebra $\bA$ of finite type
let $\Atf$ denote the abelian category of $\N_0$-graded left $\bA$-modules
of finite type. In particular, if $0\to M\to N\to Q\to 0$ is a short exact sequence
in $\Atf$, one has 
\begin{equation}
\label{eq:conn1}
h_{N}(y)=h_{M}(y)+h_{Q}(y).
\end{equation}
For $M\in\ob(\Atf)$ put $M_{\bA}=M/\bA^+M$.
Then $M_{\bA}$ is an $\N_0$-graded vector space of finite type, and one has 
a canonical projection $\pi_{M}\colon M\to M_{\bA}$ of $\N_0$-graded left $\bA$-modules.
Moreover, $M=0$ if, and only if, $M_{\bA}=0$.

Let $\sigma\colon M_{\bA}\to M$ be a section of $\pi_{M}$
in the category of $\N_0$-graded vector spaces, i.e., $\pi_{M}\circ\sigma=\iid_{M_{\bA}}$.
The map $\sigma$ induces a surjective homomorphism 
of $\N_0$-graded, left $\bA$-modules $\tilde{\sigma}\colon \bA\otimes M_{\bA}\to M$, where $\otimes=\otimes_{\F}$,
given by $\tilde{\sigma}(a\otimes m)=a\sigma(m)$, $a\in\bA$, $m\in M_{\bA}$.   
Moreover, if $M_j=0$ for $0\leq j\leq k$, then
$\kernel(\tilde{\sigma})_i=0$ for $0\leq i\leq k+1$.

The just mentioned procedure can be used to build up a projective resolution
$(P_\bullet,\der_\bullet,\eps)$ of $\F$, i.e., one may define
\begin{itemize}
\item[(i)] $P_0=\bA$, 
\item[(ii)] $P_s=\bA\otimes\kernel(\der_{s-1})_\bA$ for $s\geq 1$, where we put $\der_0=\eps$.
\end{itemize}
The mappings $\partial_s\colon P_s\to\kernel(\partial_{s-1})$, $s\geq 1$, coincide with $\tilde{\sigma}_s$, where
$\sigma_s\colon \kernel(\der_{s-1})_\bA\to\kernel(\der_{s-1})$ is a section in the category of 
$\N_0$-graded vector spaces.
For an $\N_0$-graded, connected algebra $\bA$ its {\it cohomology algebra}
$\Ext^{\bullet,\bullet}_{\bA}(\F,\F)$
is naturally bigraded. The first degree in the notation will correspond to the {\it homological degree}, while the second
will refer to the {\it internal degree}. The construction of the projective resolution 
$(P_\bullet,\der_\bullet,\eps)$ has shown the following.

\begin{prop}
\label{prop:ident}
Let $\bA$ be an $\N_0$-graded, connected algebra of finite type. 
\begin{itemize}
\item[(a)] 
$\Ext^{s,t}_{\bA}(\F,\F)\simeq (\kernel(\der_{s-1})_{\bA,t})^\ast$,
where $\argu^\ast=\Hom_\F(\argu,\F)$. 
\item[(b)] $\dim(\Ext^{s,t}_{\bA}(\F,\F))<\infty$, and
$\Ext^{s,t}_{\bA}(\F,\F)=0$ for $t<s$.
\end{itemize}
\end{prop}

One can collect the information provided by Proposition~\ref{prop:ident} 
in the power series
\begin{equation}
\label{eq:powerser}
\begin{aligned}
\tchi_{\bA}(x,y)&=\sum_{s,t\geq 0}\dim(\Ext^{s,t}_{\bA}(\F,\F))\cdot (-x)^s\cdot y^t&&\in\Z\dbl x,y\dbr,\\
\chi_{\bA}(y)=\tchi_{\bA}(1,y)&=\sum_{s,t\geq 0}(-1)^s\cdot \dim(\Ext^{s,t}_{\bA}(\F,\F))\cdot y^t&&\in\Z\dbl y\dbr.
\end{aligned}
\end{equation}
The power series $\chi_{\bA}(y)\in\Z\dbl y\dbr$ will be called the {\it characteristic power series} of 
the $\N_0$-graded, connected $\F$-algebra $\bA$. 
Obviously, $\tchi_{\bA}(x,y)$ contains the complete information
on the dimensions of the bi-graded cohomology algebra $\Ext^{\bullet,\bullet}_{\bA}(\F,\F)$.
The information contained in $\chi_{\bA}(y)$ is somehow weaker, but
its importance is reflected in the following property.

\begin{prop}
\label{prop:char}
Let $\bA$ be an $\N_0$-graded, connected $\F$-algebra. Then one has
\begin{equation}
\label{eq:chareq}
\chi_{\bA}(y)\cdot h_{\bA}(y)=1\qquad \text{in $\Z\dbl y\dbr$.}
\end{equation}
\end{prop}

\begin{proof}
The acyclicity of the chain complex $(P_\bullet,\der_\bullet)$ 
and Proposition~\ref{prop:ident} imply that
\begin{equation}
\label{eq:charpower}
1=\sum_{s\geq 0} (-1)^s\cdot h_{P_s}(y)
=h_{\bA}(y)\cdot \sum_{s,t\geq 0} (-1)^s\cdot\dim(\Ext^{s,t}_{\bA}(\F,\F))\cdot y^t.
\end{equation}
This yields the claim.
\end{proof}


\subsection{Cohomological finiteness conditions}
\label{ss:cofin}
Let $\bA$ be an $\N_0$-graded, connected algebra.
The {\it cohomological dimension} $\ccd(\bA)$ of $\bA$ is defined by
\begin{equation}
\label{eq:ccd}
\ccd(\bA)=\min(\{\,n\in\N_0\mid\Ext_\bA^{n+1}(\F,M)=0\ \text{for all $M\in\ob(\Amod)$}\,\}\cup\{\infty\}),
\end{equation}
where $\Amod$ denotes the abelian category of left $\bA$-modules.
Moreover, $\bA$ is called to be {\it of type FP$_m$}, $1\leq m\leq \infty$, if 
there exists a partial projective resolution
\begin{equation}
\label{eq:FPinf}
\xymatrix{
P_m\ar[r]^{\der_m}&
P_{m-1}\ar[r]^{\der_{m-1}}&\ldots\ar[r]^{\der_2}&
P_1\ar[r]^{\der_a}&P_0\ar[r]^{\eps}&\F\ar[r]&0}
\end{equation}
with $P_0,\ldots, P_m$ being finitely generated. 
If $\bA$ is of type FP$_\infty$ and $\ccd(\bA)<\infty$, then one calls $\bA$ to be {\it of type
FP} (cf. \cite[Chap.~8.6]{brown:coho}).
From Proposition~\ref{prop:ident} one concludes the following.

\begin{fact}
\label{fact:propco}
Let $\bA$ be an $\N_0$-graded connected algebra of finite type.
\begin{itemize}
\item[(a)] $\ccd(\bA)\leq d<\infty$ if, and only if, $\Ext_{\bA}^{d+1,\bullet}(\F,\F)=0$.
\item[(b)] $\bA$ is of type FP$_\infty$ if, and only if, for every $s\geq 1$ there exists $m(s)\geq 0$
such that $\Ext_{\bA}^{s,t}(\F,\F)=0$ for all $t\geq m(s)$.
\item[(c)] $\bA$ is of type FP if, and only if, $\tilde{\chi}_{\bA}(x,y)$ is a polynomial.
\end{itemize}
\end{fact}

\begin{rem}
\label{rem:tensor}
(a) For an $\N$-graded vector space $V$ ($V_0=0$) the {\it tensor algebra}
\begin{equation}
\label{eq:tensor}
\textstyle{\boT(V)=\coprod_{k\geq 0}\boT_k(V),\qquad
\boT_0(V)=\F\cdot 1,}\qquad
\boT_k(V)=\overbrace{V\otimes\cdots\otimes V}^{\text{$k$-times}}
\end{equation} 
is a connected $\N_0$-graded algebra with the grading induced by the grading of $V$.
Moreover, one has a projective resolution
\begin{equation}
\label{eq:tensor2}
\xymatrix{
0\ar[r]&\boT(V)\otimes V\ar[r]^-{\der_1}&\boT(V)\ar[r]^-{\eps}&\F\ar[r]&0},
\end{equation}
where $\der_1$ is given by multiplication.
Thus, $\boT(V)$ is of finite type if, and only if, $V$ is of finite type,
and $\boT(V)$ is of type FP$_1$ if, and only if, $V$ is finite-dimensional.

\noindent
(b) Let $\bA$ be an $\N_0$-graded, connected algebra,
and let $\tau\colon \bA^+_{\bA}\to\bA^+$ be a section in the category of $\N$-graded vector spaces.
Then one has a unique homomorphism $\tau_\bullet\colon \boT(\bA^+_{\bA})\to\bA$
of $\N_0$-graded, connected algebras satisfying $\tau_1=\tau$.
Moreover, $\tau_\bullet$ is surjective. Hence 
$\bA$ is finitely generated if, and only if, $\bA^+_{\bA}$ is finite-dimensional.
In particular, every finitely generated $\N_0$-graded, connected algebra is of finite type.
Let $\der_1^{\bA}\colon \bA\otimes \bA^+_{\bA}\to\bA$ be given by
$\der_1^{\bA}(a\otimes b)=a\tau(b)$, $a\in\bA$, $b\in\bA^+_{\bA}$.
Then
\begin{equation}
\label{eq:pproj}
\xymatrix{
\bA\otimes \bA^+_{\bA}\ar[r]^-{\der_1^{\bA}}&
\bA\ar[r]^-{\eps_{\bA}}&\F\ar[r]&0.}
\end{equation}
is a minimal partial projective resolution, and from the minimality 
one concludes that 
$\Tor^{\bA}_1(\F,\F)\simeq \bA^+_{\bA}$.
Thus $\bA$ is finitely generated if, and only if, $\bA$ is of type FP$_1$. 

\noindent
(c) By construction, one has a commutative diagram
\begin{equation}
\label{eq:rel1}
\xymatrix{
&&&\kernel(\tau_\bullet)\ar[d]\ar@{-->}[ld]_j\ar@/_1truecm/@{..>}[2,-2]_{\rho}&&\\
&0\ar[r]&\boT(\bA^+_{\bA})\otimes\bA^+_{\bA}\ar[r]^-{\der_1}\ar[d]^{\tau_\bullet\otimes\iid}
&\boT(\bA_{\bA}^+)\ar[r]^{\eps}\ar[d]^{\tau_\bullet}&\F\ar[r]\ar@{=}[d]&0\\
0\ar[r]&\bor(\tau_\bullet)\ar[r]&
\bA\otimes\bA^+_{\bA}\ar[r]^-{\der^{\bA}_1}&\bA\ar[r]^{\eps_{\bA}}&\F\ar[r]&0}
\end{equation}
with exact rows. The left $\bA$-module $\bor(\tau_\bullet)\simeq\Tor_1^{\boT(\bA^+_{\bA})}(\bA,\F)$ 
is also called the {\it relation module} of the presentation $\tau_\bullet$.
Since $\kernel(\tau_\bullet)\subseteq\boT(\bA^+_{\bA})^+$, there exists an injective homomorphism 
of left $\boT(\bA^+_{\bA})$-modules $j\colon\kernel(\tau_\bullet)\to
\boT(\bA^+_{\bA})\otimes\bA^+_{\bA}$
making the diagram
\eqref{eq:rel1} commute. Let $\beta=(\tau_\bullet\otimes\iid)\circ j$. As $\der^{\bA}_1\circ\beta=0$, there exists
a homomorphism of left $\boT(\bA^+_{\bA})$-modules
$\rho\colon \kernel(\pi)\to\bor(\tau_\bullet)$ making the diagram \eqref{eq:rel1} commute.
This homomorphism has the following property.

\begin{prop}
\label{prop:rel}
The homomorphism $\rho\colon\kernel(\tau_\bullet)\to\bor(\tau_\bullet)$ is surjective 
and induces an isomorphism
$\brho\colon\kernel(\tau_\bullet)/\kernel(\tau_\bullet)\bA^+_{\bA}\longrightarrow\bor(\tau_\bullet)$ of left $\bA$-modules.
\end{prop}

\begin{proof}
Let $y$ be an element in $\bor(\tau_\bullet)$, and let $z$ be its canonical image in
$\bA\otimes\bA^+_{\bA}$. As $\tau_\bullet\otimes\iid$ is surjective, there exists
$w\in \boT(\bA_{\bA}^+)\otimes\bA_{\bA}^+$ such that $(\tau_\bullet\otimes\iid)(w)=z$.
By the commutativity of \eqref{eq:rel1}, $\der_1(w)\in\kernel(\tau_\bullet)$, and it is easy to verify that
$\rho(\der_1(w))=y$. Hence $\rho$ is surjective. By definition,
$\kernel(\rho)=\der_1(\kernel(\tau_\bullet\otimes\iid))=\kernel(\tau_\bullet)\bA_{\bA}^+$. 
\end{proof}

From Proposition~\ref{prop:rel} one concludes that one has isomorphisms
\begin{equation}
\label{eq:FP2}
\Tor^{\bA}_2(\F,\F)\simeq \bor(\tau_\bullet)_{\bA}\simeq\kernel(\tau_\bullet)/
(\bA^+_{\bA}\kernel(\tau_\bullet)+\kernel(\tau_\bullet)\bA^+_{\bA}).
\end{equation}
Hence $\bA$ is type FP$_2$ if, and only if, $\bA$ is finitely presented. Note that one has an isomorphism
$\Ext^{2,\bullet}_{\bA}(\F,\F)\simeq\bor(\tau_\bullet)_{\bA}^\ast$,
where $\argu^\ast=\Hom_{\F}(\argu,\F)$.
\end{rem}


\subsection{$\chi$-finite algebras}
\label{ss:chif}
An $\N_0$-graded, connected algebra $\bA$ 
will be called to be {\it $\chi$-finite}, if it is of finite type and $\chi_{\bA}(y)$ is a polynomial.
By Fact~\ref{fact:propco}, the class of $\chi$-finite algebras 
contains the class of $\N_0$-graded, connected algebras which are of finite type and of type FP.
Moreover, by Proposition~\ref{prop:char}, this class coincides with the class 
considered in \cite[\S 2]{szsh:ent}
of $\N_0$-graded, connected algebras of finite type with a linear recurrence relation. 

Let $\bA$ be a $\chi$-finite algebra. 
The integer $\deg(\bA)=\deg(\chi_{\bA}(y))$ will be called the {\it degree} of $\bA$.
Let $K_{\bA}=\Q(\chi_{\bA})$ denote the splitting field of $\chi_{\bA}(y)$ over $\Q$,
and let $\iota\colon K_{\bA}\to\C$ be a fixed complex embedding of $K_{\bA}$
in the field of complex numbers. For simplicity we may assume that $\iota$ is given by inclusion.
The
numbers $\lambda_1,\ldots,\lambda_n\in K_{\bA}\subseteq\C$, $n=\deg(\bA)$,
satisfying
\begin{equation}
\label{eq:charp1}
\chi_{\bA}(y)=(1-\lambda_1y)\cdots (1-\lambda_n y)
\end{equation}
will be called the {\it eigenvalues} of $\bA$, and 
the leading coefficient of $\chi_{\bA}(y)$ times $(-1)^n$ 
will be called the {\it conductor} $c_{\bA}$ of $\bA$,
i.e., one has $c_{\bA}=\lambda_1\cdots\lambda_n$. For $\bA\not=\F$ put
\begin{equation}
\label{eq:lmax}
\lambda_{\max}=\max\{\,|\lambda_1|,\ldots,|\lambda_n|\,\}\in\R_{>0}.
\end{equation}
Obviously, $\lambda_{\max}\geq 1$,
and one has the following property. 

\begin{prop}
\label{prop:lmaxent}
Let $\bA$, $\bA\not=\F$, be a $\chi$-finite algebra. 
Then $\boh(\bA)=\lambda_{\max}$, and
$\lambda_{\max}$ is an eigenvalue of $\bA$. Moreover, one has $1\leq \boh(\bA)<\infty$.
\end{prop}

\begin{proof}
Let $h=h_{\bA}(y)$ denote the Hilbert series of $\bA$,
and let $\chi=\chi_{\bA}(y)$ denote the characteristic polynomial of $\bA$.
Since $\chi$ is a polynomial of degree $n=\deg(\bA)\geq 1$,
$\chi$ can be interpreted as a holomorphic function $\chi\colon \C\to\C$.
Thus, by \eqref{eq:chareq}, $h=h_{\bA}$ defines a meromorphic function
$h\colon\C\to\bar{\C}$. In particular, since $\lambda_1^{-1},\cdots,\lambda_n^{-1}\in\C$
are the roots of $\chi$, they are also the poles of $h$.
Hence, if $\lambda_j$ is an eigenvalue of $\bA$ of maximal absolute value,
$\lambda_j^{-1}$ is a pole of $h$ closest to $0$.
This implies that $\lambda_{\max}^{-1}=|\lambda^{-1}_j|=\rho$ coincides with the
convergence radius of the power series $h$, i.e., 
$\lambda_{\max}=\limsup_{n\to\infty}\sqrt[n]{\dim(\bA_n)}=\boh(\bA)$.
Since all coefficients of the power series $h$ are non-negative, one has
\begin{equation}
\label{eq:ineq1}
|h(z)|\leq\textstyle{\sum_{k\geq 0} \dim(\bA_k)\cdot |z|^k=h(|z|)}
\end{equation}
for all $z\in \C$ with $|z|<\rho$. Hence $\rho=|\lambda_j^{-1}|$ is a pole
of $h$ and thus must coincide with one of the elements 
$\lambda_1^{-1},\cdots,\lambda_n^{-1}$, i.e., there exists $i\in\{1,\ldots,n\}$ such
that $\lambda_i=|\lambda_j|=\rho^{-1}=\boh(\bA)$. 
As $\lambda_1\cdots\lambda_n=c_{\bA}$, one has $\lambda_{\max}\geq 1$.
\end{proof}

\begin{prop}
\label{prop:lmax1}
Let $\bA$, $\bA\not=\F$, be a $\chi$-finite algebra 
satisfying $\boh(\bA)=1$. Then
all eigenvalues $\lambda_1,\ldots,\lambda_n$, $n=\deg(\bA)$, are roots of unity.
\end{prop}

\begin{proof}
By definition, $K_{\bA}/\Q$ is a Galois extension. Let $G=\Gal(K_{\bA}/\Q)$
denote its Galois group. 
In particular, $G$ acts on the set $\{\,\lambda_1,\ldots,\lambda_n\}$.
Let $\iota_k\colon K_{\bA}\to\C$, $1\leq k\leq r$, $\iota_1=r$, denote the different embeddings of $K_{\bA}$ 
into the field 
of complex numbers, and let $|\argu|_j=|\argu|\circ\iota_j\colon K_{\bA}\to\R_{\geq 0}$ denote the 
associated absolute values. Then one has a homomorphism of groups
\begin{equation}
\label{eq:abs}
\textstyle{\beta=\coprod_{1\leq k\leq r}|\argu|_j\colon K^\times_{\bA}\longrightarrow\coprod_{1\leq k\leq r} \R_{>0}},
\end{equation}
where $K^\times_{\bA}$ denotes the multiplicative group of the field $K_{\bA}$.

By construction, $f(y)=y^n\chi_{\bA}(1/y)\in\Z[y]$ and $f(\lambda_j)=0$ for all $j\in\{1,\ldots,n\}$.
Hence $\lambda_j\in\caO$, where $\caO$ denotes the integral closure of $\Z$ in $K_{\bA}$.
As $\lambda_1\cdots\lambda_n=c_{\bA}$ and $\lambda_{\max}=1$, one has $|c_{\bA}|\leq 1$. 
Since $c_{\bA}$ is a non-trivial integer,
this implies that $c_{\bA}\in\{\pm 1\}$, and therefore, $|\lambda_j|=1$ for all $j\in\{1,\ldots,n\}$.
Moreover, as $\lambda_1\cdots\lambda_n\in\{\pm 1\}$, one has that $\lambda_j\in\caO^\times$,
where $\caO^\times$ denotes the group of units in the ring $\caO$.
Since $K_{\bA}/\Q$ is a Galois extension, for any $k\in\{1,\ldots,r\}$ there exists $g_k\in G$ such that
$\iota_k=\iota\circ g_k$. Hence
\begin{equation}
\label{eq:muK}
|\lambda_j|_k=|g_k(\lambda_j)|=1.
\end{equation}
Hence $\lambda_j\in\kernel(\beta)\cap\caO^\times=\mu(K_{\bA})$,
where $\mu(K_{\bA})$ denotes the group of roots of unity in the number field $K_{\bA}$
(cf. \cite[Chap. I, \S 7, Thm.~1]{neu:alg}). This yields the claim.
\end{proof}

\begin{rem}
\label{rem:FP}
Let $\bA$ be an $\N_0$-graded, connected algebra of finite type which is of type FP$_\infty$.
Then there is also
another type of power series one studies in this context, i.e., 
the power series
\begin{equation}
\label{eq:poinser}
p_{\bA}(x)=\sum_{s\geq 0} (-1)^s\cdot\dim(H^s(\bA,\F))\cdot x^s\in\Z\dbl x\dbr
\end{equation}
is called the {\it Poincar\'e series} of $\bA$, i.e., one has $p_{\bA}(x)=\tchi_{\bA}(x,1)$
(cf. \eqref{eq:powerser}). If $\bA$ is additionally of type FP, then $p_{\bA}(1)=\chi_{\bA}(1)$
is also called the {\it Euler-Poincar\'e characteristic} of $\bA$, i.e., $\bA$ has Euler-Poincar\'e 
characteristic $0$ if, and only if, $1$ is an eigenvalue of $\bA$.
\end{rem}


\subsection{Graded algebras generated in degree $1$}
\label{ss:gendeg1}
Let $V$ be a graded vector space concentrated in degree $1$, i.e., $V_s=0$ for $s\not=1$.
The $\N_0$-graded algebra $\boT(V)$
has the property that for every $\N_0$-graded algebra $\bA$
and for every homomorphism of vector spaces $\phi\colon V\to\bA_1$
there exists a unique homomorphism of $\N_0$-graded algebras
$\phi_\bullet\colon\boT(V)\to\bA$ such that $\phi_1=\phi$.
The $\N_0$-graded algebra $\bA$ is said to be {\it generated in degree $1$}, if
$\iid_\bullet\colon\boT(\bA_1)\to\bA$ is surjective.
In particular, such an $\N_0$-graded algebra must be connected,
and multiplication induces a surjective map
$\bA_m\otimes\bA_n\to\bA_{n+m}$ for all $m,n\geq 0$, i.e.,
$\dim(\bA_{m+n})\leq\dim(\bA_m)\cdot\dim(\bA_n)$
for all $m,n\geq 0$.
One has the following properties.

\begin{fact}
\label{fact:onegen}
Let $\bA$ be an $\N_0$-graded, connected algebra.
\begin{itemize}
\item[(a)] $\bA$ is generated in degree $1$ if, and only if, $\Ext^{1,t}_{\bA}(\F,\F)=0$ for all $t\geq 2$.
\item[(b)] Suppose $\bA$ is generated in degree $1$. Then $\bA$ is finitely generated if, and only if,
it is of finite type.
\item[(c)] If $\bA$ is finitely generated and generated in degree $1$, then
$\lim_{k\to\infty} \sqrt[k]{\dim(\bA_k)}$ exists and is equal to $\boh(\bA)$.
Moreover, $\boh(\bA)\leq \dim(\bA_1)$.
\end{itemize}
\end{fact}

\begin{proof} (a) and (b) are straightforward. For (c) see
 \cite[Part I, Prob.~98, p.~23]{psz:analysis}. 
\end{proof}


\subsection{Quadratic algebras}
\label{ss:quad}
Let $\bA$ be an $\N_0$-graded algebra which is generated in degree $1$.
Then $\bA$ is said to be {\it quadratic}, if
\begin{equation}
\label{eq:quad}
\kernel(\iid_\bullet)=\langle\bor_2(\bA)\rangle=
\boT(\bA_1)\otimes \bor_2(\bA)\otimes\boT(\bA_1),
\end{equation}
where $\bor_2(\bA)=\kernel(\iid_2)\subseteq\bA_1\otimes\bA_1=\boT_2(\bA_1)$.
From Remark \ref{rem:tensor}(c) one concludes the following.

\begin{fact}
\label{fact:quad}
Let $\bA$ be an $\N_0$-graded, connected algebra which is generated
in degree $1$. Then $\bA$ is quadratic if, and only if, $\Ext^{2,t}_{\bA}(\F,\F)=0$ for all $t\geq 3$.
\end{fact}

Let $\bA$ be a quadratic algebra of finite type, and put
\begin{equation}
\label{eq:quaddual1}
\bor_2(\bA)^\bot=\{\,c\in\bA_1^\ast\otimes\bA_1^\ast\mid\langle c,a\rangle=0\ \text{for all $a\in \bor_2(\bA)$}\,\},
\end{equation}
where $\argu^\ast=\Hom_{\F}(\argu,\F)$, and
$\langle\argu,\argu\rangle\colon (\bA_1^\ast\otimes\bA_1^\ast)\otimes\bA_1\otimes\bA_1\longrightarrow \F$
denotes the evaluation homomorphism.
Then $\bA^!=\boT(\bA_1^\ast)/\langle \bor_2(\bA)^\bot\rangle$ is a quadratic algebra which is
called the {\it quadratic dual of $\bA$}. By construction, one has a
natural isomorphism $(\bA^{!})^!\simeq\bA$.


\subsection{Koszul algebras}
\label{ss:koszul}
A quadratic algebra of finite type $\bA$ is said to be {\it Koszul},
if $\Ext^{s,t}_{\bA}(\F,\F)=0$ for $s\not=t$. By definition, $\bA$ is of type FP$_\infty$,
and $\chi_{\bA}(y)=p_{\bA}(y)$ (cf. \eqref{eq:poinser}).
Koszul algebras were introduced by S.B.~Priddy in \cite{prid:kosz}, where he showed that
for such algebras
one has an isomorphism
\begin{equation}
\label{eq:kosz1}
\bA^!\simeq\diag(\Ext_{\bA}^{\bullet,\bullet}(\F,\F));
\end{equation}
in particular, $\chi_{\bA}(y)=p_{\bA}(y)=h_{\bA^!}(-y)$.
An $\F$-Koszul algebra is of finite cohomological dimension if, and only if,
$\bA$ is $\chi$-finite. In this case one has $\deg(\bA)=\ccd(\bA)$.
If $\bA$ is an $\F$-Koszul algebra of finite cohomological dimension $d=\deg(\bA)$, then
\begin{equation}
\chi_{\bA}(y)=p_{\bA}(y)=1-b_1\cdot y+b_2\cdot y^2+\cdots (-1)^d b_d\cdot y^d
\end{equation}
for positive integers $b_j\geq 1$.
Hence by R.~Descartes' rule of signs one concludes the following:

\begin{fact}
\label{fact:desrule}
Let $\bA$ be an Koszul algebra of finite cohomological dimension.
Then every real eigenvalue of $\bA$ must be positive.
In particular, $-1$ is not an eigenvalue of $\bA$.
\end{fact}

There exist Koszul algebras of finite cohomological dimension
with non-real eigenvalues (cf. Ex.~\ref{ex:RAL}(d)).
Nevertheless, the author could not find any example settling the following
question.

\begin{ques}
\label{ques:reparteig}
Does there exist a Koszul algebra of finite cohomological dimension
with an eigenvalue $\lambda$ satisfying $\Repa(\lambda)<0$?
\end{ques}

By definition, every quadratic $\F$-algebra $\bA$ of finite type 
satisfying $\ccd(\bA)\leq 2$ is Koszul. Such an algebra has two eigenvalues.
By Proposition~\ref{prop:lmaxent}, one of it is a positive real number. 
Hence the other is real as well, and, by Fact~\ref{fact:desrule}, it is also positive.
From this fact one concludes the following.

\begin{cor}
\label{cor:golod}
Let $\bA$ be a 
Koszul algebra of finite cohomological dimension less or equal to $2$.
Then the eigenvalues are positive real numbers, and 
\begin{equation}
\label{eq:gosha}
\dim( \Ext_{\bA}^{2,2}(\F,\F))\leq \frac{\dim(\Ext^{1,1}_{\bA}(\F,\F))^2}{4}.
\end{equation}
\end{cor}

Hence Koszul algebra of cohomological dimension less or equal to $2$
satisfies the Golod-Shafarevich inequality (cf. \cite[\S I, App.~2]{ser:gal}) in the opposite direction.
Corollary~\ref{cor:golod} applied to right-angled Artin algebras (cf. \S\ref{sss:RAL}) yields
an alternative proof of W.~Mantel's theorem on the number of edges in a triangle free graph
(cf. \cite{mant:gr}). This result is a special case of a more general result due to P.~Tur\'an
(cf. \cite{turan:gr}). The natural question is whether there exists an analogue of Tur\'an's theorem
also in the context of Koszul algebras.

\begin{ques}
\label{ques:turan}
Let $\bA$ be a Koszul algebra of cohomological dimension $d$. Is it true that
\begin{equation}
\label{eq:gosha2}
\dim( \Ext_{\bA}^{2,2}(\F,\F))\leq \frac{d-1}{2d}\cdot\dim(\Ext^{1,1}_{\bA}(\F,\F))^2?
\end{equation}
\end{ques}


\section{Graded Lie algebras of finite type}
\label{s:grlie}
Let $\bL=\coprod_{k\geq 1}\bL_k$ be an $\N$-graded Lie algebra,
i.e., $[\bL_n,\bL_m]\subseteq\bL_{n+m}$. Then its universal eneveloping algebra
$\caU(\bL)$ is an $\N_0$-graded,
connected algebra. Moreover, $\bL$ is of finite type if, and only if, $\caU(\bL)$ is of finite type.
We will say that $\bL$ has one of the properties $\euX$ discussed in 
previous section, if $\caU(\bL)$ has the property $\euX$,
e.g., $\bL$ is said to be {\it generated in degree $1$}, if $[\bL,\bL]=\coprod_{k\geq 2}\bL_k$,
and $\bL$ is said to be {\it quadratic} (resp. {\it Koszul}), if $\caU(\bL)$ is quadratic (resp. Koszul).
Hence, if $\bL$ is generated in degree $1$, one has
\begin{equation}
\label{eq:gen1lie}
\bL_k=\overbrace{[\bL_1,[\bL_1,\ldots,[\bL_1,\bL_1]]]}^{\text{$k$-times}}
\end{equation}
In particular, if $\bL$ is of finite type, generated in degree $1$ and $\bL_k=0$ for some $k>1$, then $\bL$
must be finite-dimensional.
If $\caU(\bL)$ is $\chi$-finite (cf. \S \ref{ss:chif}), we will call $\bL$ to be $\chi$-finite,
and will say that the eigenvalues $\lambda_1,\ldots,\lambda_n$, $n=\deg(\caU(\bL))$, of $\caU(\bL)$ are also the eigenvalues of $\bL$. In this case we also put $\deg(\bL)=\deg(\caU(\bL))$.
By Fact~\ref{fact:propco}(c), if $\bL$ is of type FP, then it is also $\chi$-finite.
For an $\N$-graded Lie algebra of finite type, we put $\chi_{\bL}(y)=\chi_{\caU(\bL)}(y)\in\Z\dbl y\dbr$
(cf. \eqref{eq:powerser}).


\subsection{The entropy of a graded Lie algebra}
\label{ss:entropy}
Let $\bL$ be an $\N$-graded Lie algebra of finite type. The
{\it entropy}\footnote{Following \cite{griha:growth} one may consider $\boh(\bL)$ also as the {\it exponential growth rate}
of $\bL$.}
$\boh(\bL)$ of $\bL$ is defined by
\begin{equation}
\label{eq:entLie}
\boh(\bL)=\limsup_{k\to\infty} \sqrt[k]{\dim(\bL_k)}.
\end{equation}
In \cite[Lemma~1]{ber:subexp}, A.E.~Berezny\u\i\ stated the following lemma.

\begin{lem}[A.E.~Berezny\u\i]
\label{lem:brez}
Let $\bL$ be an $\N$-graded Lie algebra of finite type such that $\ell_k=\dim(\bL_k)\geq 1$. Then
$\boh(\bL)=\boh(\caU(\bL))$.
\end{lem}

As A.E.~Berezny\u\i\ gives a hint, but no complete proof of the lemma stated above, for the convenience of the reader 
we provide a short proof based on the argument given in \cite[Proof of Thm.~1]{kkm:inter}.
The proof will make use of the following property.

\begin{prop}
\label{prop:sum}
Let $(a_j)_{j\geq 1}$ be a sequence of positive integers, and let $s_k=\sum_{1\leq j\leq k} a_j$.
If $\limsup_{k\to\infty}\sqrt[k]{a_k}=\lambda<\infty$, then
$\limsup_{k\to\infty}\sqrt[k]{s_k}=\lambda$.
\end{prop}

\begin{proof}
It suffices to show that $\limsup_{k\to\infty} \sqrt[k]{s_k}\leq\lambda$.
By hypothesis, $\lambda\geq 1$. For a given $\varepsilon>0$ one has
$a_k\leq (\lambda+\varepsilon)^k$ for all $k\geq N(\varepsilon)$.
Hence there exists $c(\varepsilon)\in\R_{>0}$ such that
\begin{equation}
\label{eq:lemsum1}
\begin{aligned}
s_k&\leq c(\eps)+\sum_{1\leq j\leq k} (\lambda+\eps)^j=
c(\eps)+\frac{(\lambda+\eps)^{k+1}-(\lambda+\eps)}{\lambda-1+\eps}\\
&\leq c(\eps)+\frac{\lambda+\eps}{\lambda-1+\eps}\cdot (\lambda+\eps)^k.
\end{aligned}
\end{equation}
Hence $\limsup _{k\to\infty} \sqrt[k]{s_k}\leq \lambda+\eps$, and this yields the claim.
\end{proof}

\begin{proof}[Proof of Lemma~\ref{lem:brez}]
Let $h_{\bL}(y)=\sum_{k\geq 1}\ell_k\cdot y^k$ denote the Hilbert series of $\bL$,
and let $h_{\bA}(y)$ denote the Hilbert series
of $\bA=\caU(\bL)$. Let $\sum_{k\geq 1} b_k\cdot y^k=\log(h_{\bA}(y))$
be the formal power series obtained by substituting $u=1-h_{\bA}(y)$ in
the formal power series $\log(1-u)=-\sum_{j\geq 1} \frac{u^j}{j}$.
Then, as $\exp(y)=\sum_{j\geq 0} y^j/j!$ has convergence radius equal to $\infty$,
the formal power series $\log(h_{\bA}(y))$ and $h_{\bA}(y)$ have the same
convergence radius, i.e., one has
$\boh(\bA)=\limsup_{k\to\infty} \sqrt[k]{b_k}$.
From the Poincar\'e-Birkhoff-Witt theorem (cf. \cite[\S1.2]{kkm:inter}), one concludes easily that
$b_k=\sum_{d\vert k} \frac{\ell_{k/d}}{d}$. In particular,
$\ell_k\leq b_k\leq \big(\sum_{d\vert k} \frac{1}{d}\big)\cdot s_k$,
where $s_k=\sum_{1\leq j\leq k}\ell_j$.
Moreover, by the unconditional Robin inequality (cf. \cite{robin:ineq}), one has
$\limsup_{k\to\infty}\sqrt[k]{\sum_{d\vert k} \frac{1}{d}}=1$.
Hence Proposition~\ref{prop:sum} yields the claim.
\end{proof}

Applying Lemma~\ref{lem:brez} to $\N$-graded Lie algebras of finite type
which are generated in degree $1$ one obtains the following.

\begin{cor}
\label{cor:brez1}
Let $\bL\not=0$ be a finitely generated Lie algebra which is generated in degree $1$.
\begin{itemize}
\item[(a)] If $\dim(\bL)<\infty$, then $\boh(\bL)=0$ and $\boh(\caU(\bL))=1$.
\item[(b)] If $\dim(\bL)=\infty$, then $\boh(\bL)=\boh(\caU(\bL))\geq 1$.
\end{itemize}
\end{cor}

In \cite{szsh:ent}, the authors implicitly assume that every
graded Lie algebra $\bL$ under consideration is infinite dimensional. 
This is the reason why case (a) of Corollary~\ref{cor:brez1} does never occur in their paper.


\subsection{A generalized Witt formula}
\label{ss:witt}
The {\it necklace polynomial}\footnote{The number of 
necklaces of length $k$ made from $r$-colured beads was first 
studied by Col.~C.P.N.~Moreau in 1872 (cf. \cite{mor:neck}). 
The integer $M_k(r)$ equals the number of aperiodic necklaces of length $k$ made from $r$-coloured beads.}
of degree $k$, $k\geq 1$,
is the polynomial given by
\begin{equation}
\label{eq:neck1}
M_k(y)=\frac{1}{k}\sum_{j|k} \mu(k/j)\cdot y^j\in\Q[y],
\end{equation}
where $\mu\colon\N\to\Z$ denotes the {\it M\"obius function}.
E.g., one has
\begin{equation}
\label{eq:neck2}
\begin{aligned}
M_1(y)&=y,\\
M_2(y)&=\frac{1}{2}(y^2-y),\\
M_3(y)&=\frac{1}{3}(y^3-y),\\
M_4(y)&=\frac{1}{4}(y^4-y^2),\ \text{etc.}
\end{aligned}
\end{equation}
In \cite{witt:lie}, E.~Witt showed that the dimension of the $k^{th}$-homogeneous component
of a Lie algebra $\bL$ generated freely by $r$-elements is given by $M_k(r)$.
This result can be generalized in the following way.

\begin{thm}
\label{thm:witt}
Let $\bL$ be a $\chi$-finite $\N$-graded Lie algebra,
let $n=\deg(\bL)$, and let $\lambda_1,\ldots,\lambda_n$ denote
the eigenvalues of $\bL$. Then
\begin{equation}
\label{eq:witt}
\textstyle{\dim(\bL_k)=\sum_{1\leq i\leq n} M_k(\lambda_i).}
\end{equation}
\end{thm}

\begin{proof}
Let $\ell_k=\dim(\bL_k)$. By the Poincar\'e-Birkhoff-Witt theorem, \eqref{eq:chareq} and the definition
of the eigenvalues of $\bL$, one has
\begin{equation}
\label{eq:hallpf1}
h_{\caU(\bL_\bullet)}(y)=\textstyle{\prod_{k\geq 1} (1-y^k)^{-\ell_k}}=\prod_{1\leq i\leq n} (1-\lambda_i y)^{-1}.
\end{equation}
Applying $-\log(\argu)$ on both sides and using the identity
$-\log(1-u)=\sum_{j\geq 1} \frac{u^j}{j}$ for $u\in y\C\dbl y\dbr$ one obtains
\begin{equation}
\label{eq:hallpf2}
\sum_{k\geq 1} \ell_k\sum_{j\geq 1} \frac{y^{kj}}{j}=
\sum_{m\geq 1} \frac{y^m}{m}\sum_{1\leq i\leq n} \lambda_i^m=
\sum_{m\geq 1} \frac{\bp_m(\lambda_1,\ldots,\lambda_n)}{m}\cdot y^m,
\end{equation}
where $\bp_m(\lambda_1,\ldots,\lambda_n)=\sum_{1\leq i\leq n}\lambda_i^m$.
Comparing coefficients of $y^m$ on both sides yields
\begin{equation}
\label{eq:hallpf3}
\sum_{k|m} k\cdot \ell_k = \bp_m(\lambda_1,\ldots,\lambda_n).
\end{equation}
Hence, by the M\"obius inversion formula, one obtains 
\begin{equation}
\label{eq:hallpf4}
\ell_k=\frac{1}{k}\sum_{j|k}\mu(k/j)\cdot \bp_{j}(\lambda_1,\ldots,\lambda_n)=\sum_{1\leq i\leq n} M_k(\lambda_i).
\end{equation}
This yields the claim.
\end{proof}

\begin{rem}
\label{rem:witt}
(a) For the Lie algebra $\bL=L\langle \euX\rangle$, $|\euX|=r$, generated freely by $r$ elements, one has
$\chi_{\bL}(y)=1-r\cdot y$, i.e., $\deg(\bL)=1$ and $r$ is the eigenvalue.
Hence \eqref{eq:hallpf4} coincides with Witt's formula in this case.
There exists a generalization of Witt's formula in the case when $\euX$ is a graded set
(cf. \cite{kika:witt}).

\noindent
(b) Let $\bL=L\langle \euX\rangle\oplus L\langle \euY\rangle$, $|\euX|=r$, $|\euY|=s$. 
Then $\chi_{\bL}(y)=(1-r\cdot y)(1-s\cdot y)$, i.e., $\deg(\bL)=2$ and $r$ and $s$ are the eigenvalues.
For this Lie algebra there exists also an explicit formula for the multi-graded homogeneous components
(cf. \cite{dok:witt}).

\noindent
(c) Let $\bL$ be a finite-dimensional $\N$-graded Lie algebra.
Put $m_k=\textstyle{\sum_{j\in\N}\dim(\bL_{kj})}$ for $k\geq 1$. 
By Proposition~\ref{prop:char} and \eqref{eq:hallpf1}, one has
\begin{equation}
\label{eq:mult}
\chi_{\bL}(y)=\textstyle{\prod_{k\geq 1} \Phi_k(y)^{m_k},}
\end{equation}
where $\Phi_k(y)$ denotes the $k^{th}$-cyclotomic polynomial of degree $\varphi(k)$,
and $\varphi\colon\N\to\Z$ denotes Euler's $\varphi$-function, i.e., all eigenvalues of
$\bL$ are roots of unity.
From \eqref{eq:hallpf1} 
one concludes that
\begin{equation}
\label{eq:degLiefin}
\textstyle{\deg(\bL)=\sum_{k\geq 1} k\cdot\dim(\bL_k).}
\end{equation}

\noindent
(d) Let $\bL$ be a filiform Lie algebra. Then
\begin{equation}
\label{eq:fili1}
\chi_{\bL}(y)=h_{\caU(\bL)}(y)^{-1}=(1-y)\cdot \phi(y),
\end{equation}
where $\phi(y)$ is Euler's function. This Lie algebra is an example where $\boh(\bL)=1$ holds,
but the formal power series $\chi_{\bL}(y)$ cannot be continued meromorphically to the whole complex plane.
\end{rem}


\subsection{Necklace polynomials at roots of unity}
\label{ss:neckroot}
For a positive integer $m$ let $\Xi_m\subseteq\C^\ast$
denote the set of primitive $m^{th}$-roots of unity in the field of complex numbers.
The aim of this subsection is to compute the values of the functions
$P_k\colon \N\to\C$, $C_k\colon \N\to\C$, $k\geq 1$, where
\begin{equation}
\label{eq:defC}
\begin{aligned}
P_k(m)&=\textstyle{\sum_{\xi\in\Xi_m} \xi^k,}\\
C_k(m)&=\textstyle{\sum_{\xi\in\Xi_m} M_k(\xi)},
\end{aligned}
\end{equation}
for $m\in\N$. For a positive integer $k\in\N$ we 
define a function $\delta_k\colon\N\to\C$ by
\begin{equation}
\label{eq:delta}
\delta_k(m)=
\begin{cases}
m&\ \text{if $m\vert k$,}\\
0&\ \text{if $m\!\!\not\vert k$.}
\end{cases}
\end{equation}
Obviously, $\delta_k(1)=1$ for all $k\geq 1$.
If $m_1$ and $m_2$ are positive coprime integers, then $m_1m_2$ divides $k$ if, and only if,
$m_1$ divides $k$ and $m_2$ divides $k$. Hence $\delta_k$ is a multiplicative arithmetic function.
Moreover, $\delta_1$ coincides with the unit under convolution ``$\ast$''.
The following proposition gives a complete description of the function
$P_k$, $k\geq 1$.

\begin{prop}
\label{prop:Cs}
Let $k$ be a positive integer.
\begin{itemize}
\item[(a)] $P_k$ is a multiplicative arithmetic function.
\item[(b)] $P_1=C_1=\mu$, where $\mu\colon \N\to\Z$ is the M\"obius function.
\item[(c)] For $k\geq 1$ one has $P_k=\delta_k\ast\mu$.
\item[(d)] Let $m\geq 1$ with $\gcd(m,k)=1$. Then $P_k(m)=\mu(m)$.
\item[(e)] Let $m\geq 1$ with $\gcd(m,k)=1$ and assume that $k>1$. Then $C_k(m)=0$.
\end{itemize}
\end{prop}

\begin{proof}
(a) By definition, $P_k(1)=1$ for all $k\geq 1$. 
If $m_1$ and $m_2$ are positive coprime integers, one has $\Xi_{m_1m_2}=\Xi_{m_1}\Xi_{m_2}$.
Thus $P_k(m_1m_2)=P_k(m_1)\cdot P_k(m_2)$ showing that $P_k$ is multiplicative.

\noindent
(b) For every prime number $p$ one has 
$\Phi_p(y)=\sum_{0\leq j\leq p-1} y^j$, where $\Phi_p(y)\in\Z[y]$ denotes the $p^{th}$-cyclotomic polynomial. 
Hence
$P_1(p)=-1$. 
Moreover, for $\alpha\geq 2$ and $p$ prime, the 
$(p^\alpha)^{th}$-cyclotomic
polynomial $\Phi_{p^\alpha}(y)\in\Z[y]$ is given by
\begin{equation}
\label{eq:cycl}
\Phi_{p^\alpha}(y)=\frac{y^{p^\alpha}-1}{y^{p^{\alpha-1}}-1}=y^{(p-1)p^{\alpha-1}}+
y^{(p-2)p^{\alpha-1}}+\cdots +y^{p^{\alpha-1}}+1.
\end{equation}
Hence $P_1(p^\alpha)=\sum_{\xi\in\Xi_{p^\alpha}}\xi=0$, and $P_1$ coincides with $\mu$ on all prime powers.

\noindent
(c) Let $k\geq 2$, and let $p^\alpha$ be a prime power.
We proceed by a case-by-case analysis.

\noindent
{\bf Case 1:} $\gcd(k,p^\alpha)=1$. In this case $\argu^k\colon\Xi_{p^\alpha}\to\Xi_{p^\alpha}$ is a bijection.
Hence $P_k(p^\alpha)=P_1(p^\alpha)=\mu(p^\alpha)$.
On the other hand, for $\gcd(k,p^\alpha)=1$ one has
\begin{equation}
\label{eq:neck3}
(\delta_k\ast\mu)(p^\alpha)=\sum_{0\leq j\leq \alpha} \delta_k(p^j)\mu(p^{\alpha-j})=\mu(p^\alpha). 
\end{equation}
\noindent
{\bf Case 2:} $k=p^\gamma\cdot \beta$, $\gcd(\beta,p)=1$, $\alpha\leq \gamma$.
In this case one has $\xi^k=1$ for all $\xi\in\Xi_{p^\alpha}$, and thus $P_k(p^\alpha)=\varphi(p^\alpha)$.
On the other hand, 
\begin{equation}
\label{eq:neck4}
(\delta_k\ast\mu)(p^\alpha)=\sum_{0\leq j\leq \alpha} \delta_k(p^j)\mu(p^{\alpha-j})=
\sum_{0\leq j\leq \alpha} p^j\mu(p^{\alpha-j})=
\varphi(p^\alpha). 
\end{equation}
\noindent
{\bf Case 3:} $k=p^\gamma\cdot \beta$, $\gcd(\beta,p)=1$, $\gamma< \alpha$.
As in Case 1, $\argu^\beta\colon\Xi_{p^\alpha}\to\Xi_{p^\alpha}$ is a bi\-jection, and 
$\argu^{p^\gamma}\colon \Xi_{p^\alpha}\to\Xi_{p^{\alpha-\gamma}}$ is surjective with
all fibers of cardinality $p^\gamma$. Hence $P_k(p^\alpha)=p^\gamma\mu(p^{\alpha-\gamma})$, i.e.,
$P_k(p^{\gamma+1})=-p^\gamma$, and $P_k(p^\alpha)=0$ for $\alpha>\gamma+1$.
On the other hand, for $0\leq j\leq \alpha$, one has $\mu(p^{\alpha-j})=0$ unless $j=\alpha$ or $j=\alpha-1$.
Hence
\begin{equation}
\label{eq:neck5}
(\delta_k\ast\mu)(p^\alpha)=\sum_{0\leq j\leq \alpha} \delta_k(p^j)\mu(p^{\alpha-j})=
-\delta_k(p^{\alpha-1})+\delta_k(p^\alpha). 
\end{equation}
By hypothesis, $\delta_k(p^\alpha)=0$. Moreover, $\delta_k(p^{\alpha-1})\not=0$ if, and only if, $\gamma=\alpha-1$.
This yields the claim.

\noindent
(d) By hypothesis,
$P_k(m)=\sum_{d|m}\delta_k(d)\mu(m/d)=\mu(m)$.

\noindent
(e) By hypothesis and (d),
\begin{equation}
\label{eq:neck6}
C_k(m)=\frac{1}{k}\sum_{j|k} \mu(j)\cdot P_{k/j}(m)\\
=\frac{1}{k}\mu(m)\sum_{j|k} \mu(j) =\frac{1}{k}\mu(m)\cdot\delta_1(k)=0,
\end{equation}
and hence the claim.
\end{proof}


\subsection{Graded Lie algebras with $\boh(\caU(\bL))=1$}
\label{ss:grhU1}
The main purpose of this subsection is to prove the following theorem
and to discuss its consequences.

\begin{thm}
\label{thm:grhU1}
Let $\bL$ be an $\N$-graded Lie algebra of finite type such that
\begin{itemize}
\item[(i)] $\bL$ is generated in degree $1$;
\item[(ii)] $\bL$ is $\chi$-finite;
\item[(iii)] $\boh(\caU(\bL))=1$.
\end{itemize}
Then $\bL$ is finite-dimensional and $\boh(\bL)=0$.
\end{thm}

\begin{proof} By hypothesis (ii), $\caU(\bL)$ is $\chi$-finite. 
Let $n=\deg(\bL)$, and let $\lambda_1,\ldots,\lambda_n$ denote the eigenvalues of $\bL$.
Then, by hypothesis (iii) and Proposition~\ref{prop:lmax1}, $\lambda_i$ is a root of unity.
Let $m_i=\ord(\lambda_i)$ denote its order in the multiplicative group $\C^\times$,  let
$\caM(\bL)=\{\,m_i\mid 1\leq i\leq n\,\}$,
and let $n_i=|\{\,1\leq j\leq n\mid \lambda_j=\lambda_i\,\}|$ denote their multiplicities.
Let $\Xi_m\subseteq\C^\times$ denote the set of primitive $m^{th}$-roots of unity of $\C$.
Since the absolute Galois group $G_{\Q}$ of $\Q$ acts on the set $\Lambda=\{\,\lambda_i\mid 1\leq i\leq n\,\}$,
one has $\Xi_{m_i}\subseteq\Lambda$ for all $i\in\{1,\ldots,n\}$, and $n_i=n_j$ if $m_i=m_j$.
For $m=m_i\in \caM(\bL)$ let $n(m)=n_i$.
By Theorem~\ref{thm:witt}, one has for all $k\geq 1$ that
\begin{equation}
\label{eq:grhU1}
\dim(\bL_k)=\sum_{1\leq i\leq n} M_k(\lambda_i)=
\sum_{m\in \caM(\bL)} n(m)\cdot \sum_{\xi\in\Xi_m} M_k(\xi)
= \sum_{m\in \caM(\bL)} n(m)\cdot C_k(m)
\end{equation}
(cf. \eqref{eq:defC}).
For $k=1+\prod_{m\in\caM(\bL)}m$, one has that $\gcd(k,m)=1$ for all $m\in\caM(\bL)$.
By Proposition~\ref{prop:Cs}(e), $C_k(m)=0$ for all $m\in\caM(\bL)$, and thus
$\dim(\bL_k)=0$. As $\bL$ is $1$-generated, this shows that $\bL$ is finite-dimensional (cf. \eqref{eq:gen1lie}).
\end{proof}

From A.E.~Berezny\u\i's lemma (cf. Lemma~\ref{lem:brez}) one concludes the following:

\begin{cor}
\label{cor:brez2}
Let $\bL$ be a finitely generated $\N$-graded Lie algebra
generated in degree $1$ satisfying $\boh(\bL)=1$.
Then $\bL$ is not of type FP.
\end{cor}

An alternative reformulation of Corollary~\ref{cor:brez2} is the following.

\begin{cor}
\label{cor:brez3}
Let $\bL$ be a finitely generated $\N$-graded Lie algebra
generated in degree $1$ of type FP.
Then either $\bL$ is finite-dimensional and $\boh(\bL)=0$, or $\boh(\bL)>1$.
\end{cor}


\section{Koszul Lie algebras}
\label{s:koslie}
Let $\bL$ be a Koszul Lie algebra. Then, by definition,
$\caU(\bL)$ is $\N_0$-graded, quadratic and of finite type.
Moreover, for the cohomology algebra one has an isomorphism
\begin{equation}
\label{eq:coholie}
H^\bullet(\bL,\F)=\diag(\Ext^{\bullet,\bullet}_{\caU(\bL)}(\F,\F))\simeq\caU(\bL)^!.
\end{equation}
Let $\bL^{\ab}=\bL/\coprod_{k\geq 2} \bL_k$ denote the maximal abelian quotient of $\bL$.
Since $H^\bullet(\bL,\F)$ is a quadratic algebra, inflation 
$\iota^\bullet\colon H^\bullet(\bL^{\ab},\F)\to H^\bullet(\bL,\F)$ is a surjective homomorphism of algebras.
As $H^\bullet(\bL^{\ab},\F)$ is isomorphic to the exterior algebra $\Lambda(\bL_1^\ast)$,
one concludes the following fact 
(cf. \cite[\S7.1, Conj.~2]{popo:quad}).

\begin{fact}
\label{fact:koslie}
Let $\bL$ be a Koszul Lie algebra. Then $\ccd(\bL)\leq \dim(\bL_1)$ and equality holds if, and only if,
$\bL$ is abelian.
\end{fact}

\begin{proof}
If $\ell=\dim(\bL_1)$, then $H^{\ell+1}(\bL^{\ab},\F)=0$. As $\iota^\bullet$ is surjective, this implies $H^{\ell+1}(\bL,\F)=0$.
Hence $\ccd(\bL)\leq\ell$ (cf. Fact~\ref{fact:propco}(a)). Assume that $\ccd(\bL)=\ell$. Then
$\kernel(\iota^\ell)=0$. For any non-trivial element $x\in\Lambda_k(\bL_1^\ast)$,
there exists $y\in\Lambda_{\ell-k}(\bL^\ast_1)$ such that $x\wedge y\not=0$.
Thus $\kernel(\iota^\ell)=0$ implies that $\iota^\bullet$ is injective, and hence an isomorphism.
Therefore, $(\iota^\bullet)^!\colon\caU(\bL)\to\caU(\bL^{\ab})$ is an isomorphism, and this yields the claim.
\end{proof}


\subsection{Koszul Lie algebras with entropy equal to 1}
\label{ss:koslient1}
From Theorem~\ref{thm:grhU1} one concludes the following result
which is again an open problem for Koszul algebras in general
(cf. \cite[\S7.1, Conj.~3]{popo:quad}).

\begin{prop}
\label{prop:koslieen1}
Let $\bL$ be a Koszul Lie algebra satisfying $\boh(\caU(\bL))=1$.
Then $\bL$ is abelian.
\end{prop}

\begin{proof}
As $\caU(\bL)$ is Koszul, $\caU(\bL)$ is of type FP$_\infty$ (cf. Fact~\ref{fact:propco}(b) and \eqref{eq:kosz1}).
Thus by Fact~\ref{fact:koslie}, $\bL$ is of type FP and, therefore, $\chi$-finite (cf. Fact~\ref{fact:propco}(c)).
Hence, by Theorem~\ref{thm:grhU1}, $\bL$ is finite-dimensional.
Then
\begin{equation}
\label{eq:koslient11}
\chi_{\bL}(y)=\prod_{k\geq 1}\Phi_k(y)^{m_k}
\end{equation}
where $\Phi_k(y)$ is the $k^{th}$-cyclotomic polynomial,
and $m_k=\sum_{j\geq 1}\dim(\bL_{jk})$ (cf. Remark~\ref{rem:witt}(c)).
By Fact~\ref{fact:desrule}, $m_2=0$. Hence $\bL_2=0$, and $\bL$ is abelian.
\end{proof}


\subsection{Examples of Koszul Lie algebras}
\label{ss:exkoslie}

\subsubsection{Quadratic $1$-relator Lie algebras}
\label{sss:1rel}
Let $\euX$ be a finite set of cardinality $m\geq 2$, and let
$L\langle\euX\rangle$ denote the free $\F$-Lie algebra over the set $\euX$.
Then $L\langle\euX\rangle$ is $\N$-graded, of finite type and generated in degree $1$.
Let $\bor\in L_2\langle\euX\rangle\setminus\{0\}$ be a non-trivial homogeneous element
of degree $2$. Then putting
\begin{equation}
\label{eq:1rel1}
L\langle\,\euX\mid \bor\,\rangle=L\langle\euX\rangle/\langle\bor\rangle_{\mathrm{Lie}},
\end{equation}
where $\langle\bor\rangle_{\mathrm{Lie}}$ denotes the Lie ideal generated by $\bor$,
one has an isomorphism
\begin{equation}
\label{eq:1rel2}
\caU(L\langle\,\euX\mid \bor\,\rangle)\simeq\caU(L\langle\euX\rangle)/\langle\bor\rangle_{\mathrm{alg}},
\end{equation}
where $\langle\bor\rangle_{\mathrm{alg}}$ denotes the ideal in the associative algebra $\caU(L\langle\euX\rangle)$
generated by $\bor$. In particular, $L\langle\,\euX\mid\bor\,\rangle$ is quadratic and of finite type.
By J.~Labute's theorem (cf. \cite[Thm.~1]{labute:prop}), $\ccd(L\langle\,\euX\mid\bor\,\rangle)=2$.
Hence $L\langle\,\euX\mid\bor\,\rangle$ is a Koszul Lie algebra, 
\begin{equation}
\label{eq:1rel3}
\chi_{\bL}(y)=1-m\cdot y+y^2,
\end{equation}
and
\begin{equation}
\label{eq:1rel4}
\boh(L\langle\,\euX\mid\bor\,\rangle)=\lambda_1=\frac{1}{2}(m+\sqrt{m^2-4}),\qquad
\lambda_2=\frac{1}{2}(m-\sqrt{m^2-4});
\end{equation}
e.g., for $m=2$, $L\langle\,\euX\mid\bor\,\rangle$ is abelian. Moreover, for $\ell_k=\dim(L_k\langle\,\euX\mid\bor\,\rangle)$ one has
\begin{equation}
\label{eq:labut}
\ell_k=\frac{1}{k}\sum_{j\vert k}\mu(k/j)\sum_{0\leq i\leq j/2} (-1)^i
\frac{j}{j-i}\binom{j-i}{i}\cdot m^{j-2i}
\end{equation} 
(cf. \cite[Eq. (1)]{labut:one}). One can use \eqref{eq:labut}
to express $\bp_k(\lambda_1,\lambda_2)=\lambda_1^k+\lambda_2^k$
as
\begin{equation}
\label{eq:1rel5}
\bp_k(\lambda_1,\lambda_2)=
\sum_{0\leq i\leq k/2} (-1)^i
\frac{k}{k-i}\binom{k-i}{i}\cdot m^{k-2i}.
\end{equation}


\subsubsection{Right-angled Artin Lie algebras}
\label{sss:RAL}
Let $\Gamma=(\euX,\caE)$ be a finite loop-free graph with unoriented edges,
i.e., $|\euX|<\infty$ and $\caE\subseteq\caP_2(\euX)$, where
$\caP_2(\euX)$ is the set of subsets of $\euX$ of cardinality $2$. Then
\begin{equation}
\label{eq:RAL1}
L\langle\Gamma\rangle=L\langle\,\euX\mid xy-yx,\ \{x,y\}\in\caE\,\rangle
\end{equation}
will be called the {\it right-angled Artin Lie algebra} associated with $\Gamma$.
Moreover,
\begin{equation}
\label{eq:RAL2}
\caU(L\langle\Gamma\rangle)\simeq A\langle\Gamma\rangle=
\F\langle\euX\rangle/\langle\, xy-yx,\ \{x,y\}\in\caE\,\rangle_{\mathrm{alg}},
\end{equation}
where $\F\langle\euX\rangle$ denotes the free associative algebra over the set $\euX$.
Thus $L\langle\Gamma\rangle$ is quadratic and of finite type, and, by 
R.~Fr\"oberg's theorem (cf. \cite{froe:detpoin}), $L\langle\Gamma\rangle$ is Koszul.
For a graph $\Gamma_\circ=(\euX_\circ,\caE_\circ)$ let
the {\it exterior algebra associated with $\Gamma_\circ$} be given by
\begin{equation}
\label{eq:RAL3}
\Lambda\langle\Gamma_\circ\rangle=\Lambda\langle\euX_\circ\rangle/\langle\, x\wedge y\mid (x,y)\in\caE_\circ\,\rangle_{\mathrm{alg}},
\end{equation}
where $\Lambda\langle\euX_\circ\rangle$ denotes the free exterior algebra over the set $\euX_\circ$.
Then, by \eqref{eq:kosz1}, 
\begin{equation}
\label{eq:RAL4}
H^\bullet(L\langle\Gamma\rangle,\F)\simeq\Lambda\langle\Gamma^{\op}\rangle,
\end{equation}
where $\Gamma^{\op}=(\euX,\caP_2(\euX)\setminus \caE)$ for $\Gamma=(\euX,\caE)$.
In particular,
\begin{equation}
\label{eq:RAL5}
\chi_{L\langle\Gamma\rangle}(y)=\Cl_{\Gamma}(y)=1+\sum_{1\leq k\leq n} (-1)^k\cdot c_k(\Gamma)\cdot y^k,
\end{equation}
where $c_k(\Gamma)$ denotes the number of {\it $k$-cliques} (= complete subgraphs with $k$ vertices) 
in the graph $\Gamma$,
and $n=\ccd(L\langle\Gamma\rangle)$ coincides with the {\it clique number} of $\Gamma$, i.e.,
$\chi_{L\langle\Gamma\rangle}(-y)$ coincides with the {\it clique polynomial} of $\Gamma$ which is equal to the
{\it independence polynomial} of $\Gamma^{\op}$. 
Therefore, $\Cl_\Gamma(y)$ will be called the {\it alternating clique polynomial} of $\Gamma$.
Thus applying Corollary~\ref{cor:golod} for a
triangle-free graph $\Gamma$, one obtains Mantel's theorem. Moreover, P.~Tur\'an's theorem (cf. \cite{turan:gr}) 
shows that Question~\ref{ques:turan} has an affirmative answer for 
{\it right-angled Artin algebras} $A\langle\Gamma\rangle$ (cf.~\eqref{eq:RAL2}).

\begin{example}
\label{ex:RAL}
\noindent
(a) If $\Gamma=(\euX,\emptyset)$ has no edges, $L\langle\Gamma\rangle$ is the 
free Lie algebra over the set $\euX$, and
$\chi_{L\langle\Gamma\rangle}(y)=1-v\cdot y$, where $v=|\euX|$.

\noindent
(b) If $\Gamma=(\euX,\caP_2(\euX))$ is the complete graph on $v=|\euX|$ vertices, $L\langle\Gamma\rangle$
is abelian of dimension $v$, and $\chi_{L\langle\Gamma\rangle}(y)=(1-y)^v$.

\noindent
(c) If $\Gamma=(\euX,\caE)$ is a finite tree, $e=|\caE|$, 
one has $\chi_{L\langle\Gamma\rangle}(y)=(1-e\cdot y)(1-y)$ (cf. \cite[\S I.2, Prop.~12]{ser:trees}). 
In particular, $L\langle\Gamma\rangle$ has Euler-Poincar\'e characteristic equal to $0$
(cf. Rem.~\ref{rem:FP}).

\noindent
(d) The following example was communicated to the author by P.~Spiga.
Let $\Gamma=\Gamma(C,S)$ denote the Cayley graph for the finite cyclic group $C=\Z/11\Z$,
and let $S$ be the symmetric generating system $S=\{\,\pm 2,\pm 3,\pm 5\,\}$.
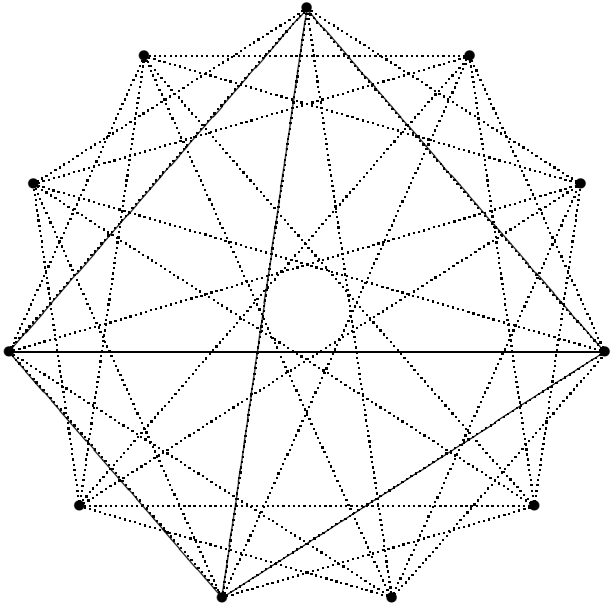
\begin{figure}[h]
$\begin{xy} 0;<4cm,0cm>:
(0,1)*=0{\bullet}="0" , 
(0.5406,0.8413)*=0{\bullet}="1",
(0.9096,0.4154)*=0{\bullet}="2",
(0.9898,-0.1423)*=0{\bullet}="3",
(0.7557,-0.6549)*=0{\bullet}="4",
(0.2817,-0.9595)*=0{\bullet} ="5",
(-0.2817,-0.9595)*=0{\bullet} ="6",
(-0.7557,-0.6549)*=0{\bullet} ="7",
(-0.9898,-0.1423)*=0{\bullet}="8",
(-0.9096,0.4154)*=0{\bullet}="9",
(-0.5406,0.8413)*=0{\bullet}="10"
\ar@{..} "0";"2",
\ar@{..} "1";"3",
\ar@{..} "2";"4",
\ar@{..} "3";"5",
\ar@{..} "4";"6",
\ar@{..} "5";"7",
\ar@{..} "6";"8",
\ar@{..} "7";"9",
\ar@{..} "8";"10",
\ar@{..} "9";"0",
\ar@{..} "10";"1",
\ar@{..} "0";"3",
\ar@{..} "1";"4",
\ar@{..} "2";"5",
\ar@{..} "3";"6",
\ar@{..} "4";"7",
\ar@{..} "5";"8",
\ar@{..} "6";"9",
\ar@{..} "7";"10",
\ar@{..} "8";"0",
\ar@{..} "9";"1",
\ar@{..} "10";"2",
\ar@{..} "0";"5",
\ar@{..} "1";"6",
\ar@{..} "2";"7",
\ar@{..} "3";"8",
\ar@{..} "4";"9",
\ar@{..} "5";"10",
\ar@{..} "6";"0",
\ar@{..} "7";"1",
\ar@{..} "8";"2",
\ar@{..} "9";"3",
\ar@{..} "10";"4"
\ar@{-} "0";"3",
\ar@{-} "3";"6",
\ar@{-} "6";"8",
\ar@{-} "8";"0",
\ar@{-} "8";"3",
\ar@{-} "6";"0"
\end{xy}$
\caption{Spiga graph}
\end{figure}
Then $\Gamma$ has clique number 4 and $\chi_{L\langle\Gamma\rangle}(y)=1-11y+33y^2-33y^3+11y^4$.
Moreover, a numerical computation of the eigenvalues shows that
\begin{equation}
\label{eq:eigenSpiga}
\begin{aligned}
\boh(L\langle\Gamma\rangle)=\lambda_1&=6.85317,\qquad&
\lambda_3&=0.751697+0.205541\,i,\\
\lambda_2&=2.64361,&
\lambda_4&=0.751697-0.205541\,i.
\end{aligned}
\end{equation}
In particular, not all eigenvalues of $L\langle\Gamma\rangle$ are real.
\end{example}

\begin{rem}
\label{rem:real}
It was shown in \cite{chsey:indep} that if $\Gamma^{\op}$ is {\it claw-free},
then all eigenvalues of $\Cl_\Gamma(y)$ are real (and positive).
Nevertheless, characterizing the finite graphs $\Gamma$ for which all eigenvalues of $\Cl_\Gamma(y)$
are real seem to be an extremely difficult problem.
\end{rem}

Let
$G_{\Gamma}=\langle\,\euX\mid xyx^{-1}y^{-1},\ \{x,y\}\in\caE\,\rangle$
denote the {\it right-angled Artin group} associated with $\Gamma$,
and let $\gr_\bullet(G_\Gamma)$ denote the graded $\Z$-Lie algebra
associated with the {\it lower central series} of $G_\Gamma$, i.e.,
$\gr_k(G_\Gamma)=\gamma_k(G_\Gamma)/\gamma_{k+1}(G_\Gamma)$,
where $\gamma_1(G_\Gamma)=G_\Gamma$ and 
$\gamma_{k+1}(G_\Gamma)=[G_\Gamma,\gamma_k(G_\Gamma)]$ for $k\geq 1$.
Then $\gr_k(G_\Gamma)$ is a torsion-free abelian group, and
$L\langle \Gamma\rangle\simeq \F\otimes_{\Z} \gr_\bullet(G_\Gamma)$
(cf. \cite{wade:RAAG}). In particular,
\begin{equation}
\label{eq:RAAG4}
\rk_{\Z}(\gr_k(G_\Gamma))=\textstyle{\sum_{1\leq j\leq n} M_k(\lambda_j),}
\end{equation}
where $n$ is the clique number of $\Gamma$ and $\lambda_1,\ldots,\lambda_n\in\C$ are the eigenvalues
of the alternating clique polynomial, i.e.,
$\Cl_\Gamma(y)=\prod_{1\leq j\leq n} (1-\lambda_j y)$.

\subsubsection{Holonomy Lie algebras of supersolvable hyperplane arrangements}
\label{sss:holo}
Let $X$ be a connected topological space having the homotopy type
of a finite cell complex, and let 
\begin{equation}
\label{eq:holo1}
\alpha_2\colon H_2(X,\F)\to\Lambda_1(H_1(X,\F))
\end{equation} 
denote
the mapping induced by the co-multiplication in $H_\bullet(X,\F)$. Then
\begin{equation}
\label{eq:holo}
L^X=L[H_1(X,\F)]/\langle\image(\alpha_2)\rangle_{\mathrm{Lie}},
\end{equation}
where $L[H_1(X,\F)]$ is the free Lie algebra over the vector space $H_1(X,\F)$
and $\image(\alpha_2)$ is identified with the corresponding subspace in $L_2[H_1(X,\F)]$,
is called the {\it holonomy Lie algebra} associated with $X$.
By definition, $L^X$ is of finite type and quadratic.

In case that $\{\,H_i\mid 1\leq i\leq s\,\}$ is a finite set of hyperplanes in the
complex vector space $\C^r$ and $X=\C^r\setminus\bigcup_{1\leq i\leq s} H_i$,
T.~Kohno has shown in \cite{kohno:holo1}
that $L^X_{\C}$ and $\C\otimes_{\Z}\gr_\bullet(\pi_1(X,x_0))$ are canonically isomorphic.
He also gave a presentation of the quadratic algebra $L^X_{\C}$.
If $\{\,H_i\mid 1\leq i\leq s\,\}$, $H_i\subseteq\C^{n+1}$ are the hyperplanes
associated with the rootsystem of type $A_n$, he showed in \cite{kohno:holo2}
that $L^X_\C$ is a Koszul Lie algebra
of cohomological dimension $n$ with eigenvalues $1$,\ldots,$n$.
In \cite{shyu:kosz}, B.~Shelton and S.~Yuzvinski extended his result to all
supersoluble hyperplane arrangements. 

\providecommand{\bysame}{\leavevmode\hbox to3em{\hrulefill}\thinspace}
\providecommand{\MR}{\relax\ifhmode\unskip\space\fi MR }
\providecommand{\MRhref}[2]{%
  \href{http://www.ams.org/mathscinet-getitem?mr=#1}{#2}
}
\providecommand{\href}[2]{#2}

\end{document}